\documentclass[11pt, a4paper]{article}
\usepackage[latin1]{inputenc}
\usepackage[T1]{fontenc}
\usepackage[english]{babel}
\usepackage{amssymb, amsmath, amsthm}
\usepackage[colorlinks=true, allcolors=blue]{hyperref}
\usepackage{geometry}
\usepackage{fancyhdr}
\usepackage{relsize}
\usepackage{upgreek}
\usepackage{verbatim}
\usepackage{enumitem}
\usepackage{calrsfs}
\usepackage{todonotes}
\usepackage{tikz}
\usetikzlibrary{matrix,arrows,calc}
\usepackage{package}
\DeclareMathOperator{\GL}{GL}
\DeclareMathOperator{\PGL}{PGL}
\newcommand{\PT}{\mathrm{P}T}
\newcommand{\fra}{\mathfrak{a}}
\newcommand{\BB}{\mathcal{B}}
\makeatletter% because of @ at preamble code
\renewcommand*{\@fnsymbol}{\@arabic}
\makeatother% because of \makeatletter
\title{Cohomology Rings of Moduli of Point Configurations on the Projective Line}
\author{Hans Franzen%
\thanks{Faculty of Mathematics, Ruhr-Universit\"at Bochum, Universit\"atsstra{\ss}e 150, 44780 Bochum\newline \href{mailto:hans.franzen@rub.de}{hans.franzen@rub.de}}
\and Markus Reineke%
\thanks{Faculty of Mathematics, Ruhr-Universit\"at Bochum, Universit\"atsstra{\ss}e 150, 44780 Bochum\newline \href{mailto:markus.reineke@rub.de}{markus.reineke@rub.de}}%
}
\date{}
\begin{document}
	\maketitle
	\begin{abstract}
		We describe the Chow rings of moduli spaces of ordered configurations of points on the projective line for arbitrary (sufficiently generic) stabilities. As an application, we exhibit such a moduli space admitting two small desingularizations with non-isomorphic cohomology rings.
	\end{abstract}
	
	\section{Introduction}
	One of the classical examples of geometric invariant theory is the moduli space of ordered point configurations on the projective line \cite[Chap.\ 3]{GIT:94}. Recall that this is the space of semi-stable ordered $m$-tuples of points in $\mathbb{P}^1$ modulo projective equivalence, that is, modulo the action of the group ${\rm PGL}_2$. Here semi-stability is typically understood with respect to the symmetric stability: a tuple of points is semi-stable if at most $m/2$ points coincide. The resulting moduli space is an irreducible normal projective variety of dimension $m-3$. It is smooth for odd $m$, and singular with isolated singularities in case $m$ is even.
	
	The (intersection) Betti numbers of this moduli space are determined in \cite[Ex.\ 8.11, 8.15]{GIT:94}. An explicit coordinatization is described in \cite{HMSV:09}. In the disguise of polygon spaces, the rational cohomology ring is described by generators and relations in \cite{HK:98}. In the case of the symmetric stability, a description of the rational Chow ring is given in \cite{Franzen:15:Chow_Ring_Quiv} using an interpretation as a moduli space of quiver representations.
	
	In the present paper, we first use the approach via moduli spaces of quiver representations (the necessary prerequisites being recalled in Section \ref{section:MQR}) to give a unified presentation of the rational Chow ring for arbitrary (sufficiently generic) stabilities in Section \ref{section:CRMPC}; see Theorem \ref{t:gens}. We first describe the rational Chow ring of the quotient stack of all ordered tuples by projective equivalence, and then determine the remaining relations arising from the open embedding of the moduli space.
	
	As our main application of this description, we show that the moduli spaces of an even number of points (with respect to symmetric stability) admit two small desingularizations with non-isomorphic rational cohomology rings, see Corollary \ref{c:non-isom}. The existence of such spaces is a classical topic of intersection homology theory, disproving the existence of a natural ring structure in intersection homology---the classical example of such a space is the Schubert variety $\{V \in \Gr_2(\C^5) \mid \dim( V \cap \C^3) \geq 1 \}$, see \cite[Ex.\ 2]{GM:83}. Using a general analysis of stability conditions in Section \ref{section:SCSQ}, we single out two stabilities deforming the symmetric one. It is shown in \cite{Reineke:15} that the corresponding moduli spaces provide small resolutions of singularities. Using the explicit description of their Chow rings, we prove the claim purely algebraically in Sections \ref{section:A} and \ref{section:RS}.%\\[2ex]
	\subsection*{Acknowledgements}
	The authors would like to thank T. Hausel and L. Migliorini for a discussion on small resolutions of moduli spaces of point configurations. While doing most of this research, H.F. was supported by the DFG SFB / Transregio 45 \glqq Perioden, Modulr\"aume und Arithmetik algebraischer Variet\"aten\grqq.

	\section{Moduli of Quiver Representations}\label{section:MQR}
	
	A quiver $Q$ is a finite oriented graph. Denote its set of vertices with $Q_0$ and its set of arrows with $Q_1$. A complex representation $M$ of $Q$ consists of complex vector spaces $M_i$ for every $i \in Q_0$ and linear maps $M_a: M_i \to M_j$ attached to every arrow $a: i \to j$. There is an obvious notion of a homomorphism of representations yielding an abelian category of all (complex) representations of $Q$. Our representations will always be assumed to be finite-dimensional (i.e.\ every $M_i$ is a finite-dimensional vector space). In this case, we can define the dimension vector $\dimvect M = (\dim M_i)_{i \in Q_0}$. Fix a dimension vector $d \in \smash{\Z_{\geq 0}^{Q_0}}$ and consider the vector space
	$$
		R(Q,d) = \bigoplus_{a: i \to j} \Hom(\C^{d_i},\C^{d_j}).
	$$
	Its elements can be regarded as representations of $Q$ of dimension vector $d$. On $R(Q,d)$ we have an action of the complex linear algebraic group $\GL_d = \prod_{i \in Q_0} \GL_{d_i}$ by change of basis. The diagonally embedded multiplicative group acts trivially whence the $\GL_d$-action descends to an action of $\PGL_d = \GL_d/\C^\times$. The orbits of this group action are in one-to-one correspondence with the isomorphism classes of (complex) representations of $Q$ of dimension vector $d$. We can also interpret this set as the set of $\C$-valued points of the quotient stack $[R(Q,d)/\PGL_d]$.
	
	If we want the quotient to carry a ``nicer'' geometric structure, we have to impose a stability condition. In King's article \cite{King:94} Mumford's criterion (cf.\ \cite[Thm.\ 2.1]{GIT:94}) is translated to a purely algebraic condition. Fix a dimension vector $d$. Let $\theta: \Q^{Q_0} \to \Q$ be a linear map for which $\theta(d) = 0$. %\todo{oder lieber slope cond'n?} 
	A representation $M$ of $Q$ of dimension vector $d$ is called $\theta$-semi-stable ($\theta$-stable) if $\theta(\dimvect M') \leq 0$ (resp.\ $\theta(\dimvect M') < 0$) for every proper, non-zero subrepresentation $M'$ of $M$. The category of $\theta$-semi-stable representations is an abelian finite length category; the simple objects of this category are the $\theta$-stable representations. We consider the open subsets $R(Q,d)^{\theta-\st} \sub R(Q,d)^{\theta-\sst} \sub R(Q,d)$. The categorical quotient $M(Q,d)^{\theta-\sst} = R(Q,d)^{\theta-\sst}/\!\!/\PGL_d$ parametrizes isomorphism classes of $\theta$-polystable representations of dimension vector $d$; these are the semi-simple objects of the category of semi-stable representations. The image of the stable locus $R(Q,d)^{\theta-\st}$ under the quotient map $R(Q,d)^{\theta-\sst} \to M(Q,d)^{\theta-\sst}$ is an open subset $M(Q,d)^{\theta-\st}$ and the restriction $R(Q,d)^{\theta-\st} \to M(Q,d)^{\theta-\st}$ is a geometric $\PGL_d$-quotient in the sense of Mumford \cite[Def.\ 0.6]{GIT:94}, even a principal fiber bundle in the \'{e}tale topology. In particular, $M(Q,d)^{\theta-\st}$ is smooth. Its points are in one-to-one correspondence with isomorphism classes of $\theta$-stable representations of dimension vector $d$. In case that the quiver has no oriented cycles the variety $M(Q,d)^{\theta-\sst}$ is projective (see \cite[Prop.\ 4.3]{King:94}).
	
	%Desingularizations via generic deformations of stability conditions after \cite{Reineke:15}.
	
	If every $\theta$-semi-stable representation of dimension $d$ is stable then the moduli spaces $M(Q,d)^{\theta-\sst}$ and $M(Q,d)^{\theta-\st}$ agree; we write $M(Q,d)^\theta$ in this case. For example, this is the case if $d$ is $\theta$-coprime, which means that $\theta(e) \neq 0$ for every dimension vector $0 \leq e \leq d$, unless $e = 0$ or $e = d$. If $d$ is $\theta$-coprime for some stability condition $\theta$ then $d$ is necessarily indivisible (i.e.\ $\operatorname{gcd}(d_i \mid i \in Q_0) = 1$). On the contrary, if $d$ is indivisible, we find a stability condition for which $d$ is coprime.
	
	However, if there are properly semi-stable points for $\theta$ then the moduli space $M(Q,d)^{\theta-\sst}$ is typically singular. The paper \cite{Reineke:15} deals with the question when small desingularizations can be constructed. 
	
	Recall that a small desingularization of a variety\footnote{Variety means irreducible here and in the following.} $X$ is a proper birational map $f: Y \to X$ from a smooth variety for which there exists a stratification $X = \bigsqcup X_i$ into locally closed subsets $X_i$ over each of which $f$ is \'{e}tale locally trivial and such that
	$$
		\dim f^{-1}(x) \leq \frac{1}{2} \codim_X(X_i)
	$$
	for every $x \in X_i$, the estimate being strict for all strata but the dense open one. The idea for constructing small desingularizations of $M(Q,d)^{\theta-\sst}$ is to find a stability condition $\theta'$ ``close to'' $\theta$ which is sufficiently generic.
	
	\begin{defn*}[{\cite[Def.\ 3.1]{Reineke:15}}]
		Let $d$ be a dimension vector of $Q$ and let $\theta$ be a stability condition such that $\theta(d) = 0$. A stability condition $\theta'$ of $Q$ with $\theta'(d) = 0$ %\todo{das ist nicht automatisch, oder?} 
		is called a deformation of $\theta$ with respect to $d$ if the following conditions hold for every proper, non-zero sub--dimension vector $0 \leq e \leq d$:
		\begin{enumerate}
			\item $\theta(e) < 0$ implies $\theta'(e) < 0$ and
			\item $\theta'(e) \leq 0$ implies $\theta(e) \leq 0$.
		\end{enumerate}
		A deformation $\theta'$ of $\theta$ with respect to $d$ is called generic if $d$ is $\theta'$-coprime. 
	\end{defn*}
	
	%The third condition means that $d$ is $\theta$-coprime, from which follows that $d$ is indivisible (i.e.\ $\operatorname{gcd}(d_i \mid i \in Q_0) = 1$). On the contrary, if $d$ is indivisible, we find a stability condition for which $d$ is coprime. 
	Under a certain assumption on the Euler form of $Q$---that is the bilinear form $\chi_Q: \Z^{Q_0} \times \Z^{Q_0} \to \Z$ defined by $\chi_Q(d,e) = \sum_{i \in Q_0} d_ie_i - \sum_{\alpha: i \to j} d_ie_j$---a small desingularization can be obtained from a generic deformation of the stability condition:
	
	\begin{thm}[{\cite[Thm.\ 4.3]{Reineke:15}}]
		Let $Q$ be a quiver, $d$ an indivisible dimension vector of $Q$, and $\theta$ a stability condition with $\theta(d) = 0$ such that a $\theta$-stable representation of dimension $d$ exists. Suppose that $\chi_Q$ is a symmetric bilinear form on $\ker(\theta)$. Then the natural morphism $p: M(Q,d)^{\theta'} \to M(Q,d)^{\theta-\sst}$ induced by a generic deformation $\theta'$ of $\theta$ is a small desingularization.
	\end{thm}

	We want to recall a description of the Chow ring of the moduli stack $[\smash{R(Q,d)^{\theta-\sst}}/\PGL_d]$, i.e.\ the equivariant Chow ring $A_{\PGL_d}^*(\smash{R(Q,d)^{\theta-\sst}})_\Q$ with rational coefficients. For simplicity we will always use rational coefficients, although it is not always necessary. Let $T_d$ be the maximal torus of $\GL_d$ that consists of invertible diagonal matrices and let $\PT_d$ be the quotient by the diagonally embedded $\C^\times$. The character group of $T_d$ is the free group generated by $x_{i,r}$ with $i \in Q_0$ and $r=1,\ldots,d_i$ and the character group of $\PT_d$ is the subgroup $X(\PT_d) = \{ \sum_{i,r} a_{i,r}x_{i,r} \mid \sum_{i,r} a_{i,r} = 0 \}$. The equivariant Chow ring $A_{\PT_d}^*(\pt)_\Q$ is the subring $\Q[x_{j,s} - x_{i,r} \mid i,j \in Q_0,\ r=1,\ldots,d_i, s=1,\ldots,d_j]$ of the polynomial ring $A_{T_d}^*(\pt)_\Q = \Q[x_{i,r} \mid i \in Q_0,\ r=1,\ldots,d_i]$. The ring $A_{\PT_d}^*(\pt)_\Q$ is itself a polynomial ring as $X(\PT_d)$ is a free abelian group but we don't want to choose a basis here.
	The equivariant Chow ring $A_{\PGL_d}^*(R(Q,d))_\Q \cong A_{\PGL_d}^*(\pt)_\Q$ agrees with the ring 
	$$
		A_{\PT}^*(\pt)_\Q^{W_d} = \Q[x_{j,s} - x_{i,r} \mid i,j \in Q_0,\ r=1,\ldots,d_i, s=1,\ldots,d_j]^{W_d}
	$$ 
	where $W_d = \prod_i S_{d_i}$ acts by permutation of the variables $x_{i,r}$; these variables are characters of a maximal torus of $\GL_d$ and $W_d$ is the corresponding Weyl group. 
	The ring $A_{\PGL_d}^*(\smash{R(Q,d)^{\theta-\sst}})_\Q$ is a quotient of $A_{\PGL_d}^*(R(Q,d))_\Q$. The kernel of $A_{\PGL_d}^*(R(Q,d))_\Q \to A_{\PGL_d}^*(\smash{R(Q,d)^{\theta-\sst}})_\Q$ can be described by tautological relations in the sense of \cite{Franzen:15:Chow_Ring_Quiv}. A $\theta$-forbidden decomposition\footnote{Note that in \cite{Franzen:15:Chow_Ring_Quiv} $\theta$-stability of a representation $M$ was defined by $\theta(M') > 0$ for every subrepresentation $M'$. We use the opposite sign convention here.} of $d$ is a decomposition $d = p+q$ into dimension vectors for which $\theta(p) > 0$. For such a decomposition we consider the element
	$$
		f^{p,q} = \prod_{\alpha: i \to j} \prod_{r=1}^{p_i} \prod_{s=p_j+1}^{d_j} (x_{j,s} - x_{i,r})
	$$
	and the principal ideal $I^{p,q} = (f^{p,q})$ in the ring $A_{\PT_d}^*(\pt)_\Q^{W_p \times W_q}$. Let
	$$
		\rho^{p,q}: A_{\PT_d}^*(\pt)_\Q^{W_p \times W_q} \to A_{\PT_d}^*(\pt)_\Q^{W_d}
	$$
	be the $A_{\PT_d}^*(\pt)_\Q^{W_d}$-linear map defined by
	$$
		\rho^{p,q}(f) = \sum_\sigma \sigma f \cdot \prod_i \prod_{r=1}^{p_i} \prod_{s=p_i+1}^{d_i} (x_{i,\sigma_i(s)} - x_{i,\sigma_i(r)})^{-1},
	$$
	where $\sigma = (\sigma_i)_i \in W_d$ ranges over all $(p,q)$-shuffles; that means each $\sigma_i$ is a $(p_i,q_i)$-shuffle permutation.
	The arguments in the proof of \cite[Thm.\ 8.1]{FR:15} show
	
	\begin{thm} \label{t:taut}
		The equivariant Chow ring $A_{\PGL_d}^*(\smash{R(Q,d)^{\theta-\sst}})_\Q$ is isomorphic to the quotient of the ring $\smash{A_{\PT_d}^*(\pt)_\Q^{W_d}}$ by the ideal
		$$
			\sum_{p,q} \rho^{p,q}(I^{p,q}).
		$$
	\end{thm}

	Note that \cite[Thm.\ 8.1]{FR:15} deals with the $\GL_d$ equivariant Chow ring of the semi-stable locus but the same arguments can be applied for the $\PGL_d$-equivariant situation as well. 
	
	\begin{rem} \label{r:cyc}
		Let us remark that analogously to \cite[Thm.\ 5.1]{FR:15}, it can be shown that the equivariant cycle map $A_{\PGL_d}^*(\smash{R(Q,d)^{\theta-\sst}})_\Q \to H_{\PGL_d}^*(\smash{R(Q,d)^{\theta-\sst}}; \Q)$ is an isomorphism. In the case that $Q$ is acyclic and $d$ is $\theta$-coprime this is shown in \cite[Thm.\ 3]{KW:95}.
	\end{rem}
	
	\section{Stability Conditions for Subspace Quivers}\label{section:SCSQ}
	
	Let $m \geq 3$. Let $U_m$ be the $m$-subspace quiver. It consists of $m$ sources and one sink and has one arrow pointing from every source to the sink. Let $d = (1,\ldots,1; 2)$, i.e.\ the dimension vector with ones at every source and two at the sink. Pictorially, the quiver setting is
	
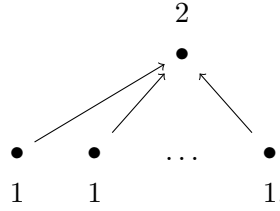
\begin{center}
	\begin{tikzpicture}[description/.style={fill=white,inner sep=2pt}]
		\matrix(m)[matrix of math nodes, row sep=2em, column sep=1.5em, text height=1.5ex, text depth=0.25ex]
		{
				&		& \bullet	& \\
			\bullet	& \bullet	& \ldots	& \bullet \\
		};
		\node[above] at (m-1-3.north) {$2$};
		\node[below] at (m-2-1.south) {$1$};
		\node[below] at (m-2-2.south) {$1$};
		\node[below] at (m-2-4.south) {\ $1$.};
		\path[->, font=\scriptsize]
		(m-2-1) edge (m-1-3)
		(m-2-2) edge (m-1-3)
		(m-2-4) edge (m-1-3);
	\end{tikzpicture}
	\end{center}
	Let $R = R(U_m,d) \cong (\C^2)^m$. The group $\GL_d$ is $(\C^\times)^m \times \GL_2(\C)$ and it acts on the vector space $R$ via $(t_1,\ldots,t_m,g).(v_1,\ldots,v_m) = (t_1^{-1}gv_1,\ldots,t_m^{-1}gv_m)$. Abbreviate $G = \PGL_d = ((\C^\times)^m \times \GL_2)/\C^\times$.
	
	We fix a stability condition $\theta$ for the $m$-subspace quiver, i.e.\ $\theta$ consists of rationals $\theta_\infty$ and $\theta_1,\ldots,\theta_m$. We assume again that $\theta(d) = 0$, so $\theta_1 + \ldots + \theta_m +2\theta_\infty = 0$. We analyze when $\theta$ is non-trivial---in the sense that the stable locus of $\theta$ is non-empty---when $d$ is coprime for $\theta$, and what it means to be a (generic) deformation.
	
	We can make a couple of reductions. We may, without loss of generality, assume that $\theta_\infty = -1$ (because a stability condition with $\theta_\infty > 0$ must yield an empty stable locus) which then means that $\theta_1 + \ldots + \theta_m = 2$. For the stable locus not to be empty, we must require that 
	$$
		\theta_i < 1. 
	$$
	This is because every representation given by the vectors $(v_1,\ldots,v_m)$ has a subrepresentation of dimension vector $(0,\ldots,1,\ldots, 0; 1)$ (with a one at the \smash{$i$\textsuperscript{th}} source) which is given by the span of the vector $v_i$. On the other hand, replacing $v_i$ with $0$ also yields a subrepresentation of $(v_1,\ldots,v_m)$ of dimension vector $(1,\ldots,0,\ldots,1; 2)$  wherefore this dimension vector must be allowed. It implies $\sum_{j \neq i} \theta_j < 2$. Adding $\theta_i$ to both sides, we get the condition
	$$
		\theta_i > 0.
	$$
	These two conditions are necessary and sufficient for the stable locus to be non-empty. So, we're considering sequences
	$\theta_1,\ldots,\theta_m$ of rational numbers $0 < \theta_i < 1$
	which sum to $2$.
	
	Let's think about coprimality. We can neglect sub--dimension vectors with $d'_\infty = 0$ or $d'_\infty = 2$ as all $\theta_i$'s are between $0$ and $1$ anyways. So we are concerned with sub--dimension vectors $d'$ with $d'_\infty = 1$. Let $I$ denote the subset of those $i$ with $d'_i = 1$. We introduce the symbol $\theta_I$ as an abbreviation for $\sum_{i \in I} \theta_i$ for any subset $I$ of $\{1,\ldots,m\}$. Then, $\theta$-coprimality of $d$ is equivalent to	
	$
		\theta_I \neq \theta_{I^c}
	$
	for every proper, non-empty subset $I \sub \{1,\ldots,m\}$, or equivalently 
	$$
		\theta_I \neq 1.
	$$
	\begin{comment}
	On the set of stability conditions of the subspace quiver with $\theta(d) = 0$, we consider the following equivalence relation: two stability conditions $\theta$ and $\theta'$ of $U_m$ are called equivalent if
	\begin{align*}
		\theta(e) < 0 &\Leftrightarrow \theta'(e) < 0, \text{ and}\\
		\theta(e) = 0 &\Leftrightarrow \theta'(e) = 0
	\end{align*}
	for every sub--dimension vector $e$ of $d$. This means that the sets of semi-stable and the sets of stable point for $\theta$ resp.\ $\theta'$ agree. By the above considerations it is clear that this is equivalent to the conditions
	\begin{align*}
		\theta_I < 1 &\Leftrightarrow \theta'_I < 1, \text{ and}\\
		\theta_I = 1 &\Leftrightarrow \theta'_I = 1
	\end{align*}
	for every proper, non-empty subset $I$ of $\{1,\ldots,m\}$. In summary:
	\end{comment}
	
	Let $\theta$ and $\theta'$ be two non-trivial stability conditions as above. Then $\theta'$ is a deformation of $\theta$ (with respect to the dimension vector $d = (1,\ldots,1;2)$) if $\theta_I < 1$ implies $\theta'_I < 1$ and if $\theta'_I \leq 1$ implies $\theta_I \leq 1$ for every proper non-empty subset $I \subseteq \{1,\ldots,m\}$.
	
	We summarize the results obtained so far in 
	 
	\begin{lem}
		Let $\theta = (\theta_1,\ldots,\theta_m;-1)$ be a stability condition of $U_m$ given by rational numbers with $\theta_1 + \ldots + \theta_m = 2$.
		\begin{enumerate}
			\item The $\theta$-stable locus is non-empty if and only if $0 < \theta_i < 1$ for every $i = 1,\ldots,m$.
			\item The dimension vector $(1,\ldots,1;2)$ is $\theta$-coprime if and only if $\theta_I := \sum_{i \in I} \theta_i \neq 1$ for every non-empty proper subset $I$ of $\{1,\ldots,m\}$.
			\item Let $\theta'$ be another such stability condition. In this case $\theta'$ is a deformation of $\theta$ with respect to $(1,\ldots,1;2)$ if and only if
			\begin{enumerate}
				\item $\theta_I < 1$ implies $\theta'_I < 1$ and
				\item $\theta'_I \leq 1$ implies $\theta_I \leq 1$
			\end{enumerate}
			for every proper non-empty subset $I \subseteq \{1,\ldots,m\}$.
		\end{enumerate}
	\end{lem}

	\begin{rem}
		Let $\theta$ be a non-trivial stability condition. For a $\theta$-semi-stable element $(v_1,\ldots,v_m) \in R$, we get $v_i \neq 0$ (as $\theta_i > 0$) for every $i  \in \{1,\ldots,m\}$. This shows that the quotient stack $[R^{\theta-\sst}/G]$ is isomorphic to the quotient stack $[((\P^1)^m)^{\theta-\sst}/\PGL_2]$ of $\theta$-semi-stable point configurations on the projective line. The same holds true for the stable moduli. In this context, a point configuration $(p_1,\ldots,p_m)$ is called $\theta$-semi-stable ($\theta$-stable) if not all $p_i$ (with $i \in I$) agree whenever $I \sub \{1,\ldots,m\}$ is a subset with $\theta_I > 1$ (resp.\ $\theta_I \geq 1$).
	\end{rem}
	
	The above remark shows that the subsets $I \sub \{1,\ldots,m\}$ with $\theta_I > 1$ play an important role. We call those subsets $\theta$-forbidden. Denote by $\mathcal{I}^\theta$ the set of all $\theta$-forbidden subsets.

	There is the so-called canonical, or symmetric, stability condition $\theta^0 = (2/m,\ldots,2/m;-1)$ for the $m$-subspace quiver. Regarding it as a moduli space of point configurations on the projective line, a configuration $(p_1,\ldots,p_m)$ is semi-stable (stable) with respect to $\theta^0$ if no more than $\lfloor m/2 \rfloor$ (resp.\ no more than $\lceil m/2 \rceil-1$) of the $p_i$'s coincide. The canonical stability condition $\theta^0$ is therefore generic if and only if $m$ is odd.
	
	We will deal with the case where $m$ is even. Say $m=2n$. We will consider two generic deformations $\theta^+$ and $\theta^-$ of $\theta^0$ given by
	\begin{align*}
		\theta^+ &= (\frac{1}{n}+\epsilon,\frac{1}{n}-\frac{\epsilon}{2n-1},\ldots,\frac{1}{n}-\frac{\epsilon}{2n-1};-1) \\
		\theta^- &= (\frac{1}{n}-\epsilon,\frac{1}{n}+\frac{\epsilon}{2n-1},\ldots,\frac{1}{n}+\frac{\epsilon}{2n-1};-1)
	\end{align*}
	for a sufficiently small rational number $\epsilon$. We analyze the forbidden subsets for these three stability conditions.
	
	\begin{lem} \label{l:forbidden}
		Let $m = 2n$. For $\epsilon$ sufficiently small $d = (1,\ldots,1;2)$ is coprime for both $\theta^+$ and $\theta^-$ and the sets of forbidden subsets for $\theta^0$, $\theta^+$ and $\theta^-$ are
		\begin{align*}
			\mathcal{I}^{\theta^0} &= \{ I \mid \lvert I \rvert > n \}, \\
			\mathcal{I}^{\theta^+} &= \mathcal{I}^{\theta^0} \sqcup \{ I \mid \lvert I \rvert = n \text{ and } 1 \in I \}, \\
			\mathcal{I}^{\theta^-} &= \mathcal{I}^{\theta^0} \sqcup \{ I \mid \lvert I \rvert = n \text{ and } 1 \notin I \}.
		\end{align*}
		As a consequence $\theta^+$ and $\theta^-$ are generic deformations of $\theta^0$.
	\end{lem}
	
	\begin{proof}
		If $I \subseteq \{1,\ldots,m\}$ is a subset with $k$ elements then the $\theta^+$-value of $I$ is
		$$
			\theta_I^+ = \begin{cases}
			             	\frac{k}{n} + \frac{2n-k}{2n-1}\epsilon & \text{if } 1 \in I \\
			             	\frac{k}{n} - \frac{k}{2n-1}\epsilon & \text{if } 1 \notin I.
			             \end{cases}
		$$
		If $\epsilon$ is smaller than $\frac{2n-1}{n(n+1)}$, which ensures that $\frac{n-1}{n} + \frac{n+1}{2n-1}\epsilon < 1$ and $\frac{n+1}{n} - \frac{n+1}{2n-1}\epsilon > 1$, then $\theta_I^+ \neq 1$ and, moreover, the set of $\smash{\theta^+}$-forbidden subsets is as asserted. The proof for $\smash{\theta^-}$ works in the same fashion.
	\end{proof}

	\section{Chow Rings of Moduli of Point Configurations}\label{section:CRMPC}
	
	We want to show how Theorem \ref{t:taut} applies to moduli of points on $\P^1$, i.e.\ moduli of representation of $U_m$ with dimension vector $d = (1,\ldots,1;2)$, cf.\ also \cite[Co.\ 29]{Franzen:15:Chow_Ring_Quiv}. The Chow ring $A_{\PT_d}^*(\pt)$ is the subring $\Q[y_j-x_i \mid i=1,\ldots,m,\ j=1,2]$ of the polynomial ring $\Q[x_1,\ldots,x_m,y_1,y_2]$. Denote $y = y_1+y_2$ and $z = y_1y_2$.
	
	\begin{lem} \label{l:Chev}
		The ring $A_{\PGL_d}^*(\pt)_\Q = \Q[y_j-x_i \mid i=1,\ldots,m,\ j=1,2]^{S_2}$ is generated by the algebraically independent elements $X_i = \smash{\frac{1}{2}}y - x_i$ (all $i=1,\ldots,m$) and $Y = \smash{\frac{1}{4}}y^2-z$.
	\end{lem}
	
	\begin{proof}
		It is obvious that $X_i$ and $Y$ are elements of the ring in question and it is easy to show that they are algebraically independent. The fact that the generating series of that ring is $(1-q)^{-m}(1-q^2)^{-1}$ concludes the proof.
	\end{proof}
	
	First of all, we consider the moduli stack $\mathcal{M}$ of $m$ points in $\P^1$ up to $\PGL_2$-action. We define $U = \{ (v_1,\ldots,v_m) \in (\C^2)^m \mid v_i \neq 0 \text{ (all $i$)} \}$. It is an open subset of $R(U_m,d) = (\C^2)^m$ and $\mathcal{M}$ is precisely the quotient stack $[U/\PGL_d]$. This shows that $A^*(\mathcal{M})_\Q$ is a quotient of $\Q[X_1,\ldots,X_m,Y]$. We show the following
	
	\begin{lem} \label{l:Chow_P1}
		The Chow ring $A^*(\mathcal{M})_\Q = A_{\PGL_2}^*((\P^1)^m)_\Q$ with rational coefficients is isomorphic to $\Q[X_1,\ldots,X_m,Y]/(X_i^2 - Y \mid i=1,\ldots,m)$. 
	\end{lem}
	
	\begin{proof}
		The proof works in the same fashion as the proof of \cite[Thm.\ 8.1]{FR:15}. We give it for completeness. The complement of $U$ inside $R = (\C^2)^m$ is the union $\bigcup_{i=1}^m Z_i$ of subspaces $Z_i = \{ (v_1,\ldots,v_m) \mid v_i = 0 \}$. The class $[Z_i]$ in the equivariant Chow ring $A_{\PGL_d}^*(R(U_m,d)) = \Q[y_j-x_i \mid i,j]$ is $(y_1-x_i)(y_2-x_i)$. We see that
		$$
			(y_1-x_i)(y_2-x_i) = z-x_iy+x_i^2 = X_i^2-Y.
		$$
		Using the fact that the map $\bigoplus_i A_{\PGL_d}^*(Z_i)_\Q \to A_{\PGL_d}^*(R(U_m,d))_\Q$ surjects onto the kernel of the map $A_{\PGL_d}^*(R(U_m,d))_\Q \to A^*(\mathcal{M})_\Q$ completes the proof.
	\end{proof}
	
	We consider a non-trivial stability condition $\theta$ given by rational numbers $\theta_1,\ldots,\theta_m$ as described in the previous section. The quotient stack $\mathcal{M}^{\theta-\sst} = [R(U_m,d)^{\theta-\sst}/\PGL_d]$ is an open substack of $\mathcal{M}$. Therefore $A^*(\mathcal{M}^{\theta-\sst})$ is a quotient of $\AA = A^*(\mathcal{M})_\Q = \Q[X_1,\ldots,X_m]/( X_i^2 = X_j^2)$.
	
	For every subset $I \sub \{1,\ldots,m\}$ we define 
	$$
		f^I = \prod_{i \in I} (y_2 - x_i);
	$$
	it is an element of $A_{\PT_d}^*(\pt)_\Q$. Let $I \in \mathcal{I}^\theta$, i.e.\ $I$ is a $\theta$-forbidden subset. For such $I$, the dimension vector $d_I = (d_{I,1},\ldots,d_{I,m};1)$---with $d_{I,1} = 1$ if $i \in I$ and $d_{I,i} = 0$ otherwise---has a positive $\theta$-value, hence gives a $\theta$-forbidden decomposition $d = d_I + (d-d_I)$. The polynomial attached to this forbidden decomposition is precisely $f^I$. Applying Theorem \ref{t:taut} yields that the kernel of $A^*(\mathcal{M})_\Q \to A^*(\mathcal{M}^{\theta-\sst})_\Q$ is generated by the elements
	\begin{align*} 
		\rho(f^I) &= \sum_{J \subsetneqq I} (-1)^{\lvert J \rvert} x_J \sum_{\nu=0}^{\lvert I-J \rvert-1} y_1^\nu y_2^{\lvert I-J \rvert-1-\nu} \\
		\rho(f^I(y_2-y_1)) &= \sum_{J \sub I} (-1)^{\lvert J \rvert} x_J (y_1^{\lvert I-J \rvert} + y_2^{\lvert I-J \rvert})
	\end{align*}
	where $\rho: A_{\PT_d}^*(\pt)_\Q \to A_{\PT_d}^*(\pt)^{S_2}$ is the symmetrization map.
	
	\begin{lem} \label{l:RS}
		Let $I \sub \{1,\ldots,m\}$ be a subset of cardinality $k$. Then 
		\begin{align*}
			\rho(f^I) &= \sum_{\nu=0}^{\lfloor \frac{k-1}{2} \rfloor} e_{k-1-2\nu}(X_i \mid i\in I)Y^\nu &
			\frac{1}{2}\rho(f^I(y_2-y_1)) &= \sum_{\nu=0}^{\lfloor \frac{k}{2} \rfloor} e_{k-2\nu}(X_i \mid i\in I)Y^\nu.
		\end{align*}
	\end{lem}
	
	\begin{proof}
		We prove the two equalities asserted in the lemma by induction on $k$. It apparently suffices to check them for the sets $I = \{1,\ldots,k\}$. Let $R_k = \rho(f^{\{1,\ldots,k\}})$ and $S_k = \frac{1}{2}\rho(f^{\{1,\ldots,k\}}(y_2-y_1))$ and denote by $\smash{\tilde{R}}_k$ resp.\ $\smash{\tilde{S}}_k$ the right-hand sides of the equations, i.e.
		\begin{align*}
			R_k &= \sum_{j=0}^{k-1} (-1)e_j(x_1,\ldots,x_k)\sum_{\nu=0}^{k-j-1}y_1^\nu y_2^{k-j-1-\nu}, &
			S_k &= \frac{1}{2} \sum_{j=0}^k (-1)^j e_j(x_1,\ldots,x_k) (y_1^{k-j} + y_2^{k-j}), \\
			\tilde{R}_k &= \sum_{\nu=0}^{\lfloor \frac{k-1}{2} \rfloor} e_{k-1-2\nu}(X_1,\ldots,X_k)Y^\nu, \text{ and} &
			\tilde{S}_k &= \sum_{\nu=0}^{\lfloor \frac{k}{2} \rfloor} e_{k-2\nu}(X_1,\ldots,X_k)Y^\nu.
		\end{align*}
		We see that $R_0 = 0 = \tilde{R}_0$ and $S_0 = 1 = \tilde{S}_1$. Obviously the expressions $\tilde{R}_k$ and $\tilde{S}_k$ satisfy the relations
		\begin{align*}
			\tilde{R}_k &= X_k\tilde{R}_{k-1}+\tilde{S}_{k-1} &
			\tilde{S}_k &= X_k\tilde{S}_{k-1}+Y\tilde{R}_{k-1}.
		\end{align*}
		To complete the proof, it suffices to show that these relations hold for $R_k$ and $S_k$ as well. This is an easy but lengthy computation. We give it here for completeness. %\todo{Sollen wir das so ausf. machen?} \enlargethispage{2em}
		\begin{align*}
			X_kR_{k-1} + S_{k-1} =& \bigg( \frac{1}{2}(y_1+y_2) - x_k \bigg)\sum_{j=0}^{k-2} (-1)^j e_j(x_1,\ldots,x_{k-1}) \sum_{\nu=0}^{k-j-2} y_1^\nu y_2^{k-j-2-\nu} \\
			& + \frac{1}{2} \sum_{j=0}^{k-1} (-1)^j e_j(x_1,\ldots,x_{k-1}) (y_1^{k-j-1} + y_2^{k-j-1}) \\
			=& \frac{1}{2} \sum_{j=0}^{k-2} (-1)^j e_j(x_1,\ldots,x_{k-1})\! \underbrace{\big( y_1^{k-j-1} +\!\!\! \sum_{\nu=0}^{k-j-2} (y_1^{\nu+1} y_2^{k-j-2-\nu} + y_1^\nu y_2^{k-j-1-\nu}) + y_2^{k-j-1} \big)}_{=\sum_{\nu=0}^{k-j-1}y_1^\nu y_2^{k-j-1-\nu}} \\
			&- \sum_{j=0}^{k-2} (-1)^j x_ke_j(x_1,\ldots,x_{k-1}) \sum_{\nu=0}^{k-j-2} y_1^\nu y_2^{k-j-2-\nu} \\
			=& \sum_{j=0}^{k-1} (-1)^j \Big( e_j(x_1,\ldots,x_{k-1}) + x_ke_{j-1}(x_1,\ldots,x_{k-1})\Big) \sum_{\nu=0}^{k-j-1} y_1^\nu y_2^{k-j-1-\nu}
		\end{align*}
		which is just $R_k$ when interpreting $e_{-1}(x_1,\ldots,x_{k-1})$ as zero. For $S_k$ the computation reads as follows:
		\begin{align*}
			X_kS_{k-1} + YR_{k-1} =& \bigg( \frac{1}{2}(y_1+y_2) - x_k \bigg) \cdot \frac{1}{2} \sum_{j=0}^{k-1} (-1)^j e_j(x_1,\ldots,x_{k-1}) (y_1^{k-j-1} + y_2^{k-j-1}) \\
			&+ \bigg( \frac{1}{4}(y_1+y_2)^2 - y_1y_2 \bigg) \sum_{j=0}^{k-2} (-1)^j e_j(x_1,\ldots,x_{k-1}) \sum_{\nu=0}^{k-j-2} y_1^\nu y_2^{k-j-2-\nu} \\
			=& \frac{1}{4} \sum_{j=0}^{k-1} (-1)^j e_j(x_1,\ldots,x_{k-1}) (y_1^{k-j} + y_1y_2(y_1^{k-j-2}+y_2^{k-j-2})+ y_2^{k-j}) \\
			&- \frac{1}{2} \sum_{j=0}^{k-1} (-1)^j x_ke_j(x_1,\ldots,x_{k-1}) (y_1^{k-j-1} + y_2^{k-j-1}) \\
			&+ \frac{1}{4} \sum_{j=0}^{k-2} (-1)^j e_j(x_1,\ldots,x_{k-1})\!\! \underbrace{\sum_{\nu=0}^{k-j-2} ( y_1^{\nu+2} y_2^{k-j-2-\nu} - 2y_1^{\nu+1}y_2^{k-j-1-\nu} + y_1^\nu y_2^{k-j-\nu} )}_{=y_1^{k-j}-y_1y_2(y_1^{k-j-2} + y_2^{k-j-2}) + y_2^{k-j}} \\
			=& \frac{1}{2} \sum_{j=0}^k (-1)^j \Big( e_j(x_1,\ldots,x_{k-1} + x_ke_{j-1}(x_1,\ldots,x_{k-1}) \Big) (y_1^{k-j} + y_2^{k-j}).
		\end{align*}
		This equals $S_k$. Here we formally need to set $e_k(x_1,\ldots,x_{k-1}) = 0$. The lemma is proved.
	\end{proof}

	We put $R_I = \rho(f^I)$ and $S_I = \frac{1}{2}\rho((y_2-y_1)f^I)$. Note also that the relations 
	\begin{align*}
		R_I &= X_iR_{I-\{i\}}+S_{I-\{i\}} &
		S_I &= X_iS_{I-\{i\}}+YR_{I-\{i\}}
	\end{align*}
	imply that we may restrict to minimal forbidden subsets $I$, i.e.\ minimal elements of $\mathcal{I}^{\theta}$ with respect to inclusion. Denote the set of those minimal forbidden subsets with $\mathcal{I}_{\min}^\theta$.
	
	\begin{thm} \label{t:gens}
		The Chow ring $A^*(\mathcal{M}^{\theta-\sst})_\Q$ is the quotient of $\mathcal{A} = \Q[X_1,\ldots,X_m,Y]/(X_i^2-Y)$ whose ideal is generated by the elements 
		\begin{align*}
			R_I &= \sum_{\nu=0}^{\lfloor \frac{k-1}{2} \rfloor} e_{k-1-2\nu}(X_i \mid i\in I)Y^\nu \text{ and} \\
			S_I &= \sum_{\nu=0}^{\lfloor \frac{k}{2} \rfloor} e_{k-2\nu}(X_i \mid i\in I)Y^\nu,
		\end{align*}
		with $I \in \mathcal{I}_{\min}^\theta$.
	\end{thm}
	
	\begin{proof}
		Applying Lemmas \ref{l:Chev}, \ref{l:Chow_P1}, and \ref{l:RS} to Theorem \ref{t:taut} yields $A^*(\mathcal{M}^{\theta-\sst})_\Q = \mathcal{A}/(R_I, S_I \mid I \in \mathcal{I}^\theta)$. The above relations show that $I \in \smash{\mathcal{I}_{\min}^\theta}$ suffice.
	\end{proof}

	We show how this applies to the stability conditions $\theta^0$, $\theta^+$ and $\theta^-$ in the case that $m = 2n$. As the semi-stable moduli stacks of $\theta^+$ and $\theta^-$ are---by genericity of the stability conditions---actually varieties, we denote them by $M^{\theta^+}$ and $M^{\theta^-}$.
	A combination of the previous theorem with Lemma \ref{l:forbidden} yields
	
	\begin{cor} \label{c:pm}
		The rings $\AA^0 = A^*(\mathcal{M}^{\theta^0-\sst})_\Q$ and $\AA^\pm = A^*(M^{\theta^\pm})_\Q$ are quotients of the ring $\AA = \Q[X_1,\ldots,X_{2n},Y]/(X_i^2-Y)$ by ideals $\fra^0$ and $\fra^\pm$ which are given by
		\begin{align*} 
			\fra^0 &= (R_I, S_I \mid I \sub \{1,\ldots,2n\},\ \lvert I \rvert = n+1), \\
			\fra^+ &= (R_I, S_I \mid I \sub \{1,\ldots,2n\},\ \lvert I \rvert = n,\ 1 \in I) + \fra^0, \\
			\fra^- &= (R_I, S_I \mid I \sub \{1,\ldots,2n\},\ \lvert I \rvert = n,\ 1 \notin I) + \fra^0.
		\end{align*}
	\end{cor}

	\section{Automorphisms}\label{section:A}
		
	Consider $\AA = \Q[X_1,\ldots,X_m,Y]/(X_i^2 - Y) = \Q[X_1,\ldots,X_m]/(X_i^2 = X_j^2)$. We want to determine the automorphism group of this graded ring. For this, consider the following automorphisms of the polynomial ring $\Q[X_1,\ldots,X_m]$.
	\begin{itemize}
		\item For a non-zero rational $d$, we denote the dilation with $d$ with $m_d$.
		\item Let $\sigma \in S_m$ be a permutation. The automorphism that sends $X_i$ to $X_{\sigma(i)}$ will be called $\pi_\sigma$.
		\item Given $i \in \{1,\ldots,m\}$, let $\tau_i$ be defined by $\tau_i(f) = f(X_1,\ldots,-X_i,\ldots,X_m)$.
	\end{itemize}
	
	We verify immediately that the above-mentioned automorphisms of the polynomial ring leave the ideal $(X_i^2-X_j^2 \mid i,j=1,\ldots,m)$ invariant. Hence they descend to automorphisms of the ring $\AA$. We denote them with the same symbol. The rest of the section will be devoted to the proof of 
		
	\begin{prop} \label{p:aut}
		If $m > 2$ then the group $\Aut(\AA)$ is generated by the elements
		\begin{itemize}
			\item $m_d$ ($d \in \Q^\times$),
			\item $\pi_\sigma$ (where $\sigma \in S_m$),
			\item $\tau_i$ ($i \in \{1,\ldots,m\}$).
		\end{itemize}
	\end{prop}
	
	\begin{proof}
		Let $\phi$ be an automorphism of $\AA$. It is given by an invertible matrix $A = (a_{ij}) \in \GL_m(\Q)$---that means $\phi(X_j) = \sum_i a_{ij}X_i$---such that $\phi(X_{j_1})^2 - \phi(X_{j_2})^2$ is contained in the ideal generated by the expressions $X_{i_1}^2 - X_{i_2}^2$. We compute
		\begin{align*}
			\phi(X_{j_1})^2 - \phi(X_{j_2})^2 =& \sum_i (a_{ij_1}^2 - a_{ij_2}^2)x_i^2 + 2\sum_{i_1 < i_2} (a_{i_1j_1}a_{i_2j_1} - a_{i_1j_2}a_{i_2j_2})x_{i_1}x_{i_2}.
		\end{align*}
		From this we deduce the relations
		\begin{enumerate}
			\item[(a)] $\sum_i a_{ij_1}^2 = \sum_i a_{ij_2}^2$
			\item[(b)] $a_{i_1j_1}a_{i_2j_1} = a_{i_1j_2}a_{i_2j_2}$
		\end{enumerate}
		for all $i_1 < i_2$ and all $j_1 < j_2$.  
		
		We assume there were an index $j$ for which the $j$\textsuperscript{th} column contains two non-zero entries, say $a_{i_1j}a_{i_2j} \neq 0$. From relation (b) we deduce that $a_{i_1j'}a_{i_2j'} \neq 0$ for every other column index $j'$ and
		$$
		a_{i_2j'} = \frac{a_{i_1j}a_{i_2j}}{a_{i_1j'}}. 
		$$
		Suppose there were a third non-zero entry $a_{i_3j}$ in the $j$\textsuperscript{th} column. We apply relation (b) for $i_1, i_3$ and $i_2, i_3$ and obtain
		\begin{align*}
			a_{i_3j'} &= \frac{a_{i_1j}a_{i_3j}}{a_{i_1j'}},  &
			a_{i_3j'} &= \frac{a_{i_2j}a_{i_3j}}{a_{i_2j'}} 
				= \frac{a_{i_1j'}a_{i_3j}}{a_{i_1j}}
		\end{align*}
		so consequently $a_{i_1j} = a_{i_1j'}$ and in a similar vein $a_{i_2j} = a_{i_2j'}$ and $a_{i_3j} = a_{i_3j'}$. As $i_3$ was chosen arbitrarily, we deduce that the $j'$\textsuperscript{th} column would have to be equal to the $j$\textsuperscript{th} which contradicts the fact that $A$ is invertible. This shows that under the assumption that there were a column which contains more than one non-zero entry, it would have to have precisely two and every other column would have precisely two non-vanishing entries in the exact same positions. This is absurd because the matrix $A$ is assumed to have at least 3 columns.
		
		Summarizing, $A$ is a matrix with at most one non-zero entry in every column. As the column sums are all the same by relation (a), we can apply a dilation to make it a matrix with entries $0$ or $\pm 1$. By regularity of $A$, every row of $A$ has also precisely one non-zero entry. Therefore, up to the application of some $\tau_i$'s, it is a permutation matrix. The proposition is proved.
	\end{proof}

	\section{The Ring Structure}\label{section:RS}
	
	This section is devoted to the proof of 
	
	\begin{thm}\label{t:non-isom}
		The rings $A^*(M^{\theta^+})_\Q$ and $A^*(M^{\theta^-})_\Q$ are not isomorphic if $n \geq 3$.
	\end{thm}
	
	\begin{proof}
		%Treat the case $n=3$ separately; we'll do that later. The idea with the 2-nilpotent elements should work here as well.
		We first treat the smallest case $n=3$ since it forms a blueprint of the proof in the general case. We abbreviate $\AA^\pm = A^*(M^{\theta^\pm})_\Q$. We claim that $\AA^-$ contains a non-zero $2$-nilpotent element which is homogeneous of degree $1$, whereas $\AA^+$ does not.
		
		Indeed, for all $2\leq i<j<k\leq 6$, we have the degree $2$ relation $Y+(X_iX_j+X_iX_k+X_jX_k)=0$ in $\AA^-$, as well as $X_i^2=Y$. Summing four of these relations, we find $4Y+2e_2(X_2,X_3,X_4,X_5)=0$, and thus $(X_2+X_3+X_4+X_5)^2=0$ as claimed.
		
		On the other hand, the relations of degree $2$ in $\AA^+$ are generated by $Y+X_1(X_i+X_j)+X_iX_j=0$ for all $2\leq i<j\leq 6$, and by $X_i^2=Y$. Assume that $x=\sum_{i=1}^6a_iX_i$ is a $2$-nilpotent homogeneous element of degree $1$. Thus
		\begin{align*}
		0=(\sum_{i=1}^6a_iX_i)^2 &= \sum_{i=1}^6a_i^2\underbrace{X_i^2}_{=Y}+2\sum_{i=2}^6a_1a_iX_1X_i+\sum_{2\leq i<j\leq 6}a_ia_j\underbrace{X_iX_j}_{=-Y-X_1(X_i+X_j)} \\
		&=(\sum_{i=1}^6a_i^2-2\sum_{2\leq i<j\leq 6}a_ia_j)Y+2\sum_{i=2}^6(a_1a_i-\sum_{\substack{j=2\\ j\not=i}}^6a_ia_j)X_1X_i,
		\end{align*}
		and hence we find the two conditions
		\begin{align*}
			\sum_{i=1}^6a_i^2 &= 2\sum_{2\leq i<j\leq 6}a_ia_j, & a_i(a_1+a_i-\sum_{j=2}^6a_j)&=0
		\end{align*}
		for all $i = 2,\ldots,6$. Let $I\subseteq \{2,\ldots,6\}$ be the set of indices $i$ for which $a_i\not=0$. If $I$ is empty, the first condition implies $a_1=0$, thus $x=0$ as claimed. Otherwise, for $i\in I$, we have $a_i=\sum_{j=2}^6a_j-a_1=:c\not=0$. Denoting $k=|I|$, we thus find $a_1=(k-1)c$, and the first condition yields $(k-1)^2c^2+kc^2=k(k-1)c^2$, a contradiction.
		
		We now turn to the general case $n \geq 4$. Using the descriptions in Corollary \ref{c:pm} we see that the generators of $\fra^\pm$ have degree at least $n-1$. We assume there were an isomorphism $\phi: \AA^+ \to \AA^-$ of graded algebras. It is induced by an automorphism of the polynomial algebra $\Q[X_1,\ldots,X_{2n},Y]$. As $n-1 \geq 3$, this isomorphism must descend to an automorphism $\phi$ of the algebra $\AA$. We read off the classification in Proposition \ref{p:aut} that $\phi$ must leave the ideal $(Y)$ invariant. The isomorphism $\phi: \AA^+ \to \AA^-$ would hence yield an isomorphism
		$$
			\phi: \AA^+/(Y) \to \AA^-/(Y).
		$$
		Abbreviate $\mathcal{B}^\pm = \AA^\pm/(Y)$. We show that the rings $\BB^+$ and $\BB^-$ can't be isomorphic.
		Both $\BB^+$ and $\BB^-$ are quotients of the ring $\Q[X_1,\ldots,X_{2n}]/(X_i^2 = 0) = \BB$. 
		The only relations of degree $n-1$ that define $\BB^+$ inside $\BB$ are
		$$
			e_{n-1}(X_1,X_{i_2},\ldots,X_{i_n}) = 0
		$$
		for all $2 \leq i_2 < \ldots < i_n \leq 2n$. This shows that a basis of the $(n-1)$\textsuperscript{st} homogeneous component of $\BB^+$ is given by the monomials $X_J = \prod_{j \in J} X_j$ with $J$ ranging over all subsets of $\{1,\ldots,2n\}$ with $1 \in J$ and $\lvert J \rvert = n-1$. A monomial $X_J$ with $J \sub \{2,\ldots,2n\}$ and $\lvert J \rvert = n-1$ can be written in terms of these monomials as
		$$
			X_J = -X_1\sum_{j \in J} X_{J-\{j\}}.
		$$
		On the other hand the degree $n-1$ part of $\BB^-$ is described inside $\BB$ by the relations
		$$
			e_{n-1}(X_{i_1},\ldots,X_{i_n}) = 0
		$$
		with $2 \leq j_1 < \ldots < j_n \leq 2n$. This is a system of $\binom{2n-1}{n}$ linearly independent equations in $\binom{2n-1}{n-1}$ variables. Therefore the monomials $X_J$ vanish in $\BB^-$ when $J \sub \{2,\ldots,2n\}$ is a subset of $n-1$ elements.
		
		We consider the Zariski-closed subsets $Z^\pm \subset \BB_1^\pm \cong \A^{2n}$ of $(n-1)$-nilpotent elements, i.e.\ $Z^\pm = \{ a \in \BB_1^\pm \mid a^{n-1} = 0 \}$. Write $a = \sum_{1=1}^{2n} a_iX_i$ as a linear combination of the basis elements. We compute
		\begin{align*}
			0 = a^{n-1} = \sum_{p_1+\ldots+p_{2n}=n-1} \binom{n-1}{p_1\ \ldots\ p_{2n}} a_1^{p_1}\ldots a_{2n}^{p_{2n}} X_1^{p_1}\ldots X_{2n}^{p_{2n}}
				= \sum_{\substack{J \sub \{1,\ldots,2n\} \\ \lvert J \rvert = n-1}} (n-1)!\, a_J X_J
		\end{align*}
		because all squares vanish in $\BB$.
		In the ring $\BB^-$ the above expression simplifies to
		$$
			0 = (n-1)!\! \sum_{\substack{K \sub \{2,\ldots,2n\} \\ \lvert K \rvert = n-2}} a_1a_K X_1X_K
		$$
		from which we see that $Z^-$ is cut out by equations $a_1a_{k_2}\ldots a_{k_{n-2}}$. The closed subset $\{a_1 = 0\}$ is an irreducible component of $Z^-$ of dimension $2n-1$.
		When working in $\BB^+$ we obtain the equation
		$$
			0 = (n-1)!\! \sum_{\substack{K \sub \{2,\ldots,2n\} \\ \lvert K \rvert = n-2}} a_K \bigg( a_1 - \sum_{j\in \{2,\ldots,2n\}-K} a_j \bigg) X_1X_K.
		$$
		Let $U \sub \BB_1^+$ be the open subset defined by $a_L \neq 0$ for all $L \subset \{2,\ldots,2n\}$ with $\lvert L \rvert = 2n-2$. The complement of $U$ is a union of hyperplanes of codimension 2. The intersection $Z^+ \cap U$ is defined by the equations
		$$
			a_1 = \sum_{\nu = 1}^{n+1} a_{j_\nu}
		$$
		where $2 \leq j_1 < \ldots < j_{n+1} \leq 2n$. These are $\binom{2n}{n+1}$ linear equations. The codimension of $Z^+ \cap U$ inside $U$ is hence at least 2. The choice of $U$ then assures that there can be no irreducible component of $Z^+$ which is of codimension 1 inside $\BB_1^+$.
		
		We have shown that $Z^+$ and $Z^-$ can't be isomorphic as varieties which shows that the rings $\BB^+$ and $\BB^-$ are non-isomorphic. This contradicts our assumption that there were an isomorphism $\AA^+ \to \AA^-$. The theorem is proved.
	\end{proof}

Combining this result with the algebraicity of cohomology (see Remark \ref{r:cyc}), we conclude:

\begin{cor}\label{c:non-isom}
The rational cohomology rings of the small desingularizations $M^{\theta^+}$ and $M^{\theta^-}$ of $M^{\theta^0-\sst}$ are not isomorphic if $n \geq 3$.
\end{cor}
	
	\bibliographystyle{abbrv}
	\bibliography{Literature}

\def\cprime{$'$}
\begin{thebibliography}{1}

\bibitem{Franzen:15:Chow_Ring_Quiv}
H.~Franzen.
\newblock Chow rings of fine quiver moduli are tautologically presented.
\newblock {\em Math. Z.}, 279(3-4):1197--1223, 2015.

\bibitem{FR:15}
H.~Franzen and M.~Reineke.
\newblock Semi-stable {C}how-{H}all algebras of quivers and quantized
  {D}onaldson-{T}homas invariants.
\newblock Preprint. ar{X}iv:1512.03748, 2015.

\bibitem{GM:83}
M.~Goresky and R.~MacPherson.
\newblock Problems and bibliography on intersection homology.
\newblock In {\em Intersection cohomology ({B}ern, 1983)}, volume~50 of {\em
  Progr. Math.}, pages 221--233. Birkh\"auser Boston, Boston, MA, 1984.

\bibitem{HK:98}
J.-C. Hausmann and A.~Knutson.
\newblock The cohomology ring of polygon spaces.
\newblock {\em Ann. Inst. Fourier (Grenoble)}, 48(1):281--321, 1998.

\bibitem{HMSV:09}
B.~Howard, J.~Millson, A.~Snowden, and R.~Vakil.
\newblock The equations for the moduli space of {$n$} points on the line.
\newblock {\em Duke Math. J.}, 146(2):175--226, 2009.

\bibitem{King:94}
A.~D. King.
\newblock Moduli of representations of finite-dimensional algebras.
\newblock {\em Quart. J. Math. Oxford Ser. (2)}, 45(180):515--530, 1994.

\bibitem{KW:95}
A.~D. King and C.~H. Walter.
\newblock On {C}how rings of fine moduli spaces of modules.
\newblock {\em J. Reine Angew. Math.}, 461:179--187, 1995.

\bibitem{GIT:94}
D.~Mumford, J.~Fogarty, and F.~Kirwan.
\newblock {\em Geometric invariant theory}, volume~34 of {\em Ergebnisse der
  Mathematik und ihrer Grenzgebiete (2)}.
\newblock Springer-Verlag, Berlin, third edition, 1994.

\bibitem{Reineke:15}
M.~Reineke.
\newblock Quiver moduli and small desingularizations of some {GIT} quotients.
\newblock Preprint. arXiv:1511.08316, 2015.

\end{thebibliography}
\end{document}